\documentclass[12pt]{article}
\usepackage[utf8]{inputenc}
\usepackage{amssymb,mathtools,cite,enumerate,color,eqnarray,hyperref,amsfonts,amsmath,amsthm,setspace,tikz,verbatim,charter,booktabs,multirow}
\usepackage{authblk}
\usepackage[a4paper,margin=1.9cm,top=2.5cm,bottom=2.5cm,centering,vcentering]{geometry}
\numberwithin{equation}{section}

\definecolor{ao(english)}{rgb}{0.0, 0.0, 0.6}

\hypersetup{colorlinks=true, linkcolor=ao(english),citecolor=ao(english)}

\usepackage[normalem]{ulem}

\theoremstyle{plain}
\newtheorem{theorem}{Theorem}
\numberwithin{theorem}{section}

\newtheorem*{corollary*}{Corollary}
\newtheorem*{Example*}{Example}

\newtheorem{lemma}[theorem]{Lemma}

\newtheorem{conjecture}[theorem]{Conjecture}
\theoremstyle{definition}

\newtheorem*{def*}{Definition}
\newtheorem*{theorem*}{Theorem}

\newtheorem*{definition*}{Definition}

\theoremstyle{remark}
\newtheorem*{remark}{Remark}

\allowdisplaybreaks

    \title{Proof of a Conjecture of Cui, Gu and Tang on 18-Colored Generalized Frobenius Partitions}

\author[1]{Hirakjyoti Das}

\author[2]{Hemjyoti Nath}

\author[3]{Abhishek Sarma}

\affil[1]{Department of Mathematics, B. Borooah College (Autonomous), Guwahati 781007, Assam, India\\
\texttt{hirak@bborooahcollege.ac.in}}

\affil[2]{Department of Mathematics, University of Florida, P.O. Box 118105, Gainesville, FL 32611-8105, USA\\
\texttt{h.nath@ufl.edu}}

\affil[3]{Department of Basic Sciences and Humanities, Assam Skill University, Mangaldai 784125, Assam, India\\
\texttt{abhitezu002@gmail.com}}

\date{}

\begin{document}

\maketitle

\begin{abstract}
Recently, the study of the number of $k$-colored generalized Frobenius partitions, denoted by $c\phi_k(n)$, has witnessed renewed interest. In this paper, we investigate congruence properties of $c\phi_{16}(n)$ and $c\phi_{18}(n)$. Our main result is a proof of the conjecture of Cui, Gu, and Tang \cite{CGT25} that, for all $n\ge0$,
$c\phi_{18}(3n+2)\equiv0\pmod{2187}$. The proof uses a $(p,k)$-parametrization together with $q$-series identities and dissections. We also establish congruences for $c\phi_{16}(n)$ modulo $1024$ and $2048$, and for $c\phi_{18}(n)$ modulo $8$ and $81$.
\end{abstract}


\medskip

	
	\section{Introduction}
    In 1984, Andrews \cite{And84} introduced the generalized Frobenius partitions of $n$, which are defined as two-rowed arrays of non-negative integers $a_j$ and $b_j$ with $a_j\ge a_{j+1}$, $b_j\ge b_{j+1}$:
    \begin{align*}
        \begin{pmatrix}
            a_1 & a_2 & \cdots & a_r\\
            b_1 & b_2 & \cdots & b_r
        \end{pmatrix}
    \end{align*}
    such that $\displaystyle{n=r+\sum_{j=1}^r}(a_j+b_j)$. In the same AMS memoir, Andrews also introduced the $k$-colored generalized Frobenius partitions of $n$, which are the generalized Frobenius partitions of $n$, whose parts appear in $k$ colors. Let $c\phi_k(n)$ denote the number of $k$-colored generalized Frobenius partitions. Andrews found that
    \begin{align*}
        \sum_{n=0}^\infty c\phi_2(n)q^n&=\dfrac{\Theta _3(q)}{(q;q)_\infty ^2} =\dfrac{(q^2;q^4)_\infty }{(q;q^2)_\infty ^4(q^4;q^4)_\infty },\\
        \sum_{n=0}^\infty c\phi_3(n)q^n&=\dfrac{1}{(q;q)_\infty ^3} \big (\Theta _3(q)\Theta _3(q^3)+\Theta _2(q)\Theta _2(q^3)\big )= \dfrac{1}{(q;q)_\infty ^3}{\left( 1+6\sum _{j=0}^\infty {\left( \dfrac{j}{3}\right) } \dfrac{q^j}{1-q^j}\right) },\\
        \sum_{n=0}^\infty c\phi_5(n)q^n&=\dfrac{1}{(q;q)_\infty ^5}{\left( 1+25\sum _{j=1}^\infty {\left( \dfrac{j}{5}\right) \dfrac{q^j}{(1-q^j)^2}} -5\sum _{j=1}^\infty {\left( \dfrac{j}{5}\right) \dfrac{jq^j}{1-q^j}}\right) },
    \end{align*}
    where for complex numbers $a$ and $q$, $\mid q\mid<1$,
    \begin{align*}
        (a;q)_\infty&:= \prod_{j=0}^\infty(1-aq^j),& f_n&:=(q^n;q^n)_\infty\\ 
        \Theta _2(q)&:=\sum _{j=-\infty }^\infty q^{(j+1/2)^2} =2q^{1/4}\dfrac{(q^4;q^4)_\infty ^2}{(q^2;q^2)_\infty }, & \Theta _3(q)&:=\sum _{j=-\infty }^\infty q^{j^2} =\dfrac{(q^2;q^2)_\infty ^5}{(q;q)_\infty ^2(q^4;q^4)_\infty ^2}
    \end{align*}
     and $\big (\frac{\cdot }{\cdot}\big )$ is the Jacobi symbol.

     Andrews' work motivated a number of researchers to investigate the generating functions and congruence properties of $c\phi_{k}(n)$. For example, in 1989 Kolitsch \cite{Kol89} obtained the generating function of $c\phi_{7}(n)$, on the other hand a few years back Baruah and Sarmah \cite{BS11,BS15} found the generating functions of $c\phi_{k}(n)$, $k\in\{4,5,6\}$. Further progress on the generating functions appeared in \cite{CWW19, JRW22, CGT24,  CGT24b, CSX25}. Lately, Cui, Gu, and Tang \cite{CGT25} worked on the integer matrix exact covering system developed by Cao \cite{Cao11} to present a uniform framework, which enabled them to find the generating functions of $c\phi_k(n)$, $k\in\{3,4,5,6,7,8,10,12,14,15,16,18,20\}$.

     Andrews himself \cite{And84} initiated the investigation of finding congruences of $c\phi_k(n)$. He found that for all $n\ge0$,
     \begin{align*}
         c\phi_2(2n+1)&\equiv 0 \pmod{2}, & c\phi_2(5n+3)&\equiv 0 \pmod{5}.
     \end{align*}
     Following his work, several authors made substantial contributions in this direction. A comprehensive overview of the developments concerning the congruence properties of $c\phi_k(n)$ has been provided by Cui, Gu, and Tang \cite[p. 779]{CGT24b}. In the same work, they also derived a generating function for $c\phi_{12}(n)$ and established three corresponding congruences. As per our knowledge, the most recent work by Wang and Wang \cite{WW25} presented an infinite family of congruences modulo arbitrary powers of 7 for $c\phi_4(n)$. 

     Following this line of renewed interest, we consider $c\phi_{16}(n)$ and $c\phi_{18}(n)$, as their congruence properties remain largely unexplored. It is worth mentioning that with increasing $k$, the generating functions associated with $c\phi_k(n)$ become more elaborate and longer. As a result, proving congruences for these functions becomes technically demanding. We establish congruences for $c\phi_{16}(n)$ modulo 1024 and 2048, and congruences for $c\phi_{18}(n)$ modulo 8, 81 and 2187. 
     
     We now state the main results of the paper.
	\begin{theorem}\label{th:1.1}
		For all $n\ge0$, we have
		\begin{align}
			c\phi_{16}\left(4n+3\right)&\equiv 0\pmod{1024},\label{eq:c_phi_1024}\\
			c\phi_{16}\left(8n+7\right)&\equiv 0\pmod{2048}.\label{eq:c_phi_2048}
		\end{align}   
	\end{theorem}

	\begin{theorem}\label{thm 18 2 power}
		For all $n\ge0$, we have
		\begin{align}
			c\phi_{18}\left(6n+5\right)&\equiv 0\pmod{8}.\label{cong: mod 8}
		\end{align}   
	\end{theorem}
	\begin{theorem}\label{thm 18 3 power}
		For all $n\ge0$, we have
		\begin{align}
			\label{c 18 81}c\phi_{18}\left(3n+1\right)&\equiv 0\pmod{81},\\
			\label{c 18 2187}c\phi_{18}\left(3n+2\right)&\equiv 0\pmod{2187}.
		\end{align}   
	\end{theorem}
    \noindent The congruence \eqref{c 18 2187} was conjectured by Cui, Gu, and Tang \cite{CGT25} and constitutes the principal result of this paper. The proof of this conjecture uses the $(p,k)$-parametrization as a central ingredient. As by-products of the proofs of Theorems \ref{thm 18 2 power} and \ref{thm 18 3 power}, we derive certain interesting $q$-product identities.

    In the next sections, we prove our theorems using $q$-series techniques. Our proofs rely on the generating functions of $c\phi_{16}(n)$ and $c\phi_{18}(n)$ by Cui, Gu, and Tang \cite{CGT25}; in particular, the proof of \eqref{c 18 2187} is completed by means of a $(p,k)$-parametrization.
	\section{Proof of Theorem \ref{th:1.1}}
	Before going to the proof, we recall the following known 2-dissections.
	\begin{lemma}\cite[Lemma 2]{matching}\label{lem-2-dis}
		We have
		\begin{align}
            f_{1}^2 &= \frac{f_{2}f_{8}^5}{f_{4}^2f_{16}^2} - 2 q \frac{f_{2} f_{16}^2}{f_{8}},\label{disf1^2}\\
            \frac{1}{f_{1}^2} &= \frac{f_{8}^5}{f_{2}^5 f_{16}^2} + 2 q \frac{f_{4}^2 f_{16}^2 }{f_{2}^5 f_{8}},\label{dis1byf1^2}\\
			f_1^4&=\dfrac{f_4^{10}}{f_2^2f_8^4}-4q\dfrac{f_2^2f_8^4}{f_4^2}\label{f_1^4},\\
			\frac{1}{f_{1}^4} &= \frac{f_{4}^{14}}{f_{2}^{14} f_{8}^4} + 4 q \frac{f_{4}^2 f_{8}^4}{f_{2}^{10}}.\label{dis1byf1^4}
		\end{align}
	\end{lemma}
    We also require the following result, which is an easy consequence of the binomial theorem. For a prime $p$, and positive integers $k$ and $\ell$, we have
\begin{align}\label{cong-modp}
	f_{k}^{p^\ell} \equiv f_{pk}^{p^{\ell-1}} \pmod{p^\ell}.
\end{align}
In the sequel, we will use the above property without any commentary. We are now ready to prove Theorem \ref{th:1.1}.

	From \cite[Theorem 3.12]{CGT25}, we have
	\begin{align}
		\sum_{n=0}^{\infty}c\phi_{16}(n)q^n&=\dfrac{1}{f^{16}_{1}} \bigg (\dfrac{f_2^{32}J_{16}}{f_1^{16}f_{32}^2} +112q\dfrac{f_2^{30}}{f_1^8f_4^7} +112q\dfrac{f_2^{30}f_4^7}{f_1^{16}f_8^6} +1792q^2\dfrac{f_2^6J_4^{23}}{f_1^8f_8^6}\nonumber \\&\quad +4480q^2\dfrac{f_2^{20}f_4^9}{f_1^{12}f_8^2} +3632q^2\dfrac{f_2^{34}f_8^2}{f_1^{16}f_4^5}-1792q^2f_2^6f_4^9+21504q^3\dfrac{f_2^{10}f_4^{11}f_8^2}{f_1^8}\nonumber \\&\quad  +8q^4\dfrac{f_2^{32}f_4^4f_{16}^3f_{32}^2}{f_1^{16}f_8^{10}} +64q^4\dfrac{f_2^{40}J_8^4f_{16}^9}{f_1^{16}f_4^{20}f_{32}^2}+32q^4\dfrac{f_2^{40}f_8^{18}f_{32}^2}{f_1^{16}f_4^{24}f_{16}^5}\nonumber \\&\quad  +4096q^4\dfrac{f_4^8f_8^{16}f_{16}}{f_2^8f_{32}^2} +12288q^4\dfrac{f_4^{27}f_8^2}{f_2^{14}}+512q^4\dfrac{f_4^{32}f_8^2f_{32}^2}{f_2^{16}f_{16}^5}\nonumber \\&\quad +1024q^4\dfrac{f_4^{36}f_{16}^9}{f_2^{16}f_8^{12}f_{32}^2}+32768q^8\dfrac{f_4^{12}f_8^6f_{16}^3f_{32}^2}{f_2^8}\bigg ).\label{gen_c_phi_16}
	\end{align}
    \begin{proof}[Proof of \eqref{eq:c_phi_1024}]
        Modulo 1024 the above equation reduces to
	\begin{align*}
		\sum_{n=0}^{\infty}c\phi_{16}(n)q^n&\equiv\dfrac{1}{f^{16}_{1}} \bigg (\dfrac{f_2^{32}J_{16}}{f_1^{16}f_{32}^2} +112q\dfrac{f_2^{30}}{f_1^8f_4^7} +112q\dfrac{f_2^{30}f_4^7}{f_1^{16}f_8^6} +768q^2\dfrac{f_2^6J_4^{23}}{f_1^8f_8^6}+384q^2\dfrac{f_2^{20}f_4^9}{f_1^{12}f_8^2}\nonumber \\&\quad  +560q^2\dfrac{f_2^{34}f_8^2}{f_1^{16}f_4^5}+256q^2f_2^6f_4^9\nonumber +8q^4\dfrac{f_2^{32}f_4^4f_{16}^3f_{32}^2}{f_1^{16}f_8^{10}} +64q^4\dfrac{f_2^{40}J_8^4f_{16}^9}{f_1^{16}f_4^{20}f_{32}^2}\nonumber \\&\quad +32q^4\dfrac{f_2^{40}f_8^{18}f_{32}^2}{f_1^{16}f_4^{24}f_{16}^5}\nonumber +512q^4\dfrac{f_4^{32}f_8^2f_{32}^2}{f_2^{16}f_{16}^5}\nonumber\bigg )\pmod{1024}.
	\end{align*}
	Employing \eqref{dis1byf1^4} in the above equation and then extracting the terms involving the odd powers of $q$, we obtain
	\begin{align*}
		\sum_{n=0}^{\infty}c\phi_{16}(2n+1)q^n&\equiv 112\frac{f_2^{119}}{f_1^{82} f_4^{38}}+32\frac{f_2^{100}f_8 }{f_1^{76} f_4^{24} f_{16}^2}+112\frac{f_2^{77}}{f_1^{54} f_4^{24}}+512q\frac{f_2^{95}}{f_1^{74} f_4^{22}}+512q\frac{f_2^{76}f_8 }{f_1^{68} f_4^8 f_{16}^2}\\&\quad+256q\frac{f_2^{53}}{f_1^{46} f_4^8}+512q\frac{f_2^{11}f_4^2 }{f_1^2}+ 256q^2\frac{f_2^{104}f_8^3 f_{16}^2 }{f_1^{76} f_4^{34}}\pmod{1024}.
	\end{align*}
	Next, Lemma \ref{lem-2-dis} helps us find that
	\begin{align*}
			\sum_{n=0}^{\infty}c\phi_{16}(4n+3)q^n&\equiv128\frac{f_1^{46}f_4^5 }{f_2^{18} f_8^2}+800\frac{f_1^{42}f_8^2 }{f_2^4 f_4^9}+128\frac{f_1^{30}f_2^6 }{f_4^3 f_8^2}+224\frac{f_1^{26}f_2^{20} f_8^2 }{f_4^{17}}+256\frac{f_1^{26}f_4^5 }{f_2^8 f_8^2}+512 f_2^7\\&\quad+512q\frac{f_1^{50}f_4^7 f_8^2 }{f_2^{28}}\\
            &\equiv800\frac{f_1^{42}f_8^2 }{f_2^4 f_4^9}+256\frac{f_1^{30}f_2^6 }{f_4^3 f_8^2}+224\frac{f_1^{26}f_2^{20} f_8^2 }{f_4^{17}}+256\frac{f_1^{26}f_4^5 }{f_2^8 f_8^2}+512 f_2^7+512q\frac{f_1^{50}f_4^7 f_8^2 }{f_2^{28}}\\
            &\equiv800\frac{f_1^{10}f_2^{12}f_8^2}{f_4^9}+224\frac{f_1^{26}f_2^{20} f_8^2 }{f_4^{17}}+512\frac{f_1^{26}f_4^5 }{f_2^8 f_8^2}+512 f_2^7+512q\frac{f_1^{50}f_4^7 f_8^2 }{f_2^{28}}\\
			&\equiv224\frac{f_2^{36} f_8^2 }{f_1^6f_4^{17}}-224\frac{f_1^{10}f_2^{12}f_8^2}{f_4^9}+512q\frac{f_4^7 f_8^2}{f_2^{3}}\pmod{1024}.
	\end{align*}
	Now, we observe that
	\begin{align}
	224\frac{ f_2^{36} f_8^2}{f_1^6 f_4^{17}}-224\frac{ f_1^{10} f_2^{12} f_8^2}{f_4^9}=
		224\frac{ f_2^{12}  f_8^2}{f_1^6 f_4^{17}}\left(f_2^{12}-f_1^8 f_4^4\right)\left(f_2^{12}+f_1^8 f_4^4\right).\label{c_phi16observe}
	\end{align}
	Rearranging the terms in \eqref{dis1byf1^4} and \eqref{f_1^4}, we have
	\begin{align}
		f_2^{12}&=\frac{f_1^4f_4^{14}}{f_2^2f_8^4}+4qf_1^4f_2^2f_4^2f_8^4,\label{modifieddis1byf1^4}\\
		f_1^8f_4^4&=\frac{f_1^4f_4^{14}}{f_2^2f_8^4}-4qf_1^4f_2^2f_4^2f_8^4,\label{modifiedf1^4}
	\end{align} respectively. Clearly \eqref{modifieddis1byf1^4} and \eqref{modifiedf1^4} give
	\begin{align}
		f_2^{12}+f_1^8f_4^4&=2\frac{f_1^4f_4^{14}}{f_2^2f_8^4},\label{cphi_16add}\\
		f_2^{12}-f_1^8f_4^4&=8qf_1^4f_2^2f_4^2f_8^4.\label{cphi_16subtract}
	\end{align}

	Invoking \eqref{cphi_16add} and \eqref{cphi_16subtract} in \eqref{c_phi16observe}, we find that
	\begin{align*}
	\sum_{n=0}^{\infty}c\phi_{16}(4n+3)q^n&\equiv 224\frac{f_2^{12}f_8^2}{f_1^6f_4^{17}}(8qf_1^4f_2^2f_4^2f_8^4)\left(2\dfrac{f_1^4f_4^{14}}{f_2^2f_8^4}\right)+512q\frac{f_4^7 f_8^2}{f_2^{3}}\\
	&\equiv 512qf_2^{19}+512qf_2^{19}\equiv 0\pmod{1024},
\end{align*}
which is
\eqref{eq:c_phi_1024}.
\end{proof}

\begin{proof}[Proof of \eqref{eq:c_phi_2048}]
    
Taking modulo 2048 in \eqref{gen_c_phi_16}, we have
\begin{align*}
    \sum_{n=0}^{\infty}c\phi_{16}(n)q^n&\equiv\dfrac{1}{f^{16}_{1}} \bigg (\dfrac{f_2^{32}f_{16}}{f_1^{16}f_{32}^2} +112q\dfrac{f_2^{30}}{f_1^8f_4^7} +112q\dfrac{f_2^{30}f_4^7}{f_1^{16}f_8^6} +1792q^2\dfrac{f_2^6J_4^{23}}{f_1^8f_8^6}+384q^2\dfrac{f_2^{20}f_4^9}{f_1^{12}f_8^2}\\&\quad+1584q^2\dfrac{f_2^{34}f_8^2}{f_1^{16}f_4^5}-1792q^2f_2^6f_4^9+1024q^3\dfrac{f_2^{10}f_4^{11}f_8^2}{f_1^8}+64q^4\dfrac{f_2^{40}J_8^4f_{16}^9}{f_1^{16}f_4^{20}f_{32}^2} \\&\quad+32q^4\dfrac{f_2^{40}f_8^{18}f_{32}^2}{f_1^{16}f_4^{24}f_{16}^5}+512q^4\dfrac{f_4^{32}f_8^2f_{32}^2}{f_2^{16}f_{16}^5} +1024q^4\dfrac{f_4^{36}f_{16}^9}{f_2^{16}f_8^{12}f_{32}^2}\bigg)\pmod{2048}.
\end{align*}
Employing Lemma \ref{lem-2-dis} in the above and then extracting the terms involving the odd powers of $q$, we obtain
	\begin{align*}
		\sum_{n=0}^{\infty}c\phi_{16}(2n+1)q^n&\equiv\frac{f_2^{32}f_{16} }{f_1^{32} f_{32}^2}+112q\frac{f_2^{30} f_4^7}{f_1^{32} f_8^6}+112q\frac{ f_2^{30}}{f_1^{24} f_4^7}+1792q^2\frac{ f_2^6 f_4^{23}}{f_1^{24} f_8^6}+4480q^2\frac{f_2^{20} f_4^9}{f_1^{28} f_8^2}\\
        &\quad +3632q^2\frac{f_2^{34} f_8^2}{f_1^{32} f_4^5}+21504q^3\frac{ f_2^{10} f_4^{11} f_8^2}{f_1^{24}} +32q^4\frac{f_2^{40} f_8^{18} f_{32}^2}{f_1^{32} f_4^{24} f_{16}^5}\\
        &\quad +8q^4\frac{f_2^{32} f_4^4 f_{16}^3 f_{32}^2 }{f_1^{32} f_8^{10}}\pmod{2048}.
	\end{align*}
	Thanks to Lemma \ref{lem-2-dis}, we have
	\begin{align*}
			\sum_{n=0}^{\infty}c\phi_{16}(4n+3)q^n&\equiv 112\frac{ f_2^{77}}{f_1^{54} f_4^{24}}+\frac{32 f_2^{68}f_8 }{f_1^{12} f_4^{24} f_{16}^2}+112\frac{f_1^{46} f_2^{55}}{f_4^{38}}+512q\frac{f_2^{11} f_4^2}{f_1^2}+256q\frac{ f_2^{33}}{f_1^6 f_4^8}\\
            &\quad +1536q\frac{ f_2^{42}f_8 }{f_4^8 f_{16}^2}+512q\frac{ f_2^{59} }{f_1^2 f_4^{22}}+1024q^2\frac{ f_2^{44}f_{16}^2  }{f_4^6 f_8^5}+256q^2\frac{ f_2^{68}f_8^3 f_{16}^2  }{f_1^4 f_4^{34}}\pmod{2048}.
	\end{align*}
   Applying Lemma \ref{lem-2-dis} in the above identity further gives
	\begin{align*}
			&\sum_{n=0}^{\infty}c\phi_{16}(8n+7)q^n\\&\equiv224\frac{f_2^{160}f_8^2}{f_1^{110} f_4^{53}}+1728\frac{f_2^{146}}{f_1^{106} f_4^{39} f_8^2}-224\frac{f_1^{34}f_2^{72} f_8^2 }{f_4^{45}}-832\frac{f_1^{38} f_2^{58}}{f_4^{31} f_8^2}+256\frac{f_1^{14} f_2^6 f_4}{f_8^2}\\&\quad+384\frac{f_1^{30} f_2^6}{f_4^3 f_8^2}+512\frac{f_1^6 f_2^2 f_4^5}{f_8^2}+1536\frac{f_1^{42} f_4}{f_2^8 f_8^2}+512\frac{f_1^{54} f_4^5}{f_2^{22} f_8^2}+1024q\frac{f_2^{136} f_8^2}{f_1^{102} f_4^{37}}\\&\quad-512q\frac{f_1^{42} f_2^{48} f_8^2}{f_4^{29}}-1024q\frac{f_1^{46} f_2^{34}}{f_4^{15} f_8^2}+1024q\frac{f_1^{58} f_4^7 f_8^2}{f_2^{32}}\\
            &\equiv224\frac{f_2^{160}f_8^2}{f_1^{110} f_4^{53}}+1728\frac{f_2^{146}}{f_1^{106} f_4^{39} f_8^2}-224\frac{f_1^{34} f_2^{72} f_8^2}{f_4^{45}}-832\frac{f_1^{38} f_2^{58}}{f_4^{31} f_8^2}+256\frac{f_1^{14} f_2^6 f_4}{f_8^2}\\&\quad+384\frac{f_1^{30} f_2^6}{f_4^3 f_8^2}+512\frac{f_1^6 f_2^2 f_4^5}{f_8^2}+1536\frac{f_1^{42} f_4}{f_2^8 f_8^2}+512\frac{f_1^{54} f_4^5}{f_2^{22} f_8^2}+1024q\frac{f_8^2 f_2^{136}}{f_1^{102} f_4^{37}}\\&\quad-512q\frac{f_1^{42} f_8^2 f_2^{48}}{f_4^{29}}-1024 q f_1^{38}+1024q f_1^{38}\\
            &\equiv224\frac{f_2^{160} f_8^2}{f_1^{110} f_4^{53}}-224\frac{f_1^{34} f_2^{72} f_8^2}{f_4^{45}}+1728\frac{f_2^{146}}{f_1^{106} f_4^{39} f_8^2}+1216\frac{f_1^{38} f_2^{58}}{f_4^{31} f_8^2}+256\frac{f_1^{14} f_2^6 f_4}{f_8^2}\\&\quad+384\frac{f_1^{30} f_2^6}{f_4^3 f_8^2}+512\frac{f_1^6 f_2^2 f_4^5}{f_8^2}+1536\frac{f_1^{42} f_4}{f_2^8 f_8^2}+512\frac{f_1^{54} f_4^5}{f_2^{22} f_8^2}+1024qf_1^2 f_4^{9}-512qf_1^2f_4^9\\
            &\equiv224\frac{ f_1^{18} f_2^8  f_8^2}{f_4^{21}}\left(f_2^{12}-f_1^8 f_4^4\right) \left(f_2^{12}+f_1^8 f_4^4\right)+1728\frac{ f_2^{146}}{f_1^{106} f_4^{39} f_8^2}+1728\frac{ f_1^6 f_2^{10} f_4}{f_8^2}\\&\quad+640\frac{ f_2^6 f_4^5}{f_1^2 f_8^2}+512\frac{ f_1^6  f_2^2f_4^5}{f_8^2}+1536\frac{ f_1^{42} f_4}{f_2^8 f_8^2}+512qf_1^2f_4^9\pmod{2048}.
	\end{align*}
    Invoking \eqref{cphi_16add} and \eqref{cphi_16subtract} in the above identity, we have
    \begin{align}
        \sum_{n=0}^{\infty}c\phi_{16}(8n+7)q^n 
        &\equiv1728\frac{ f_1^6 f_2^{10} f_4}{f_8^2}+1728\frac{ f_1^6 f_2^{10} f_4}{f_8^2}+640\frac{ f_2^6 f_4^5}{f_1^2 f_8^2}+512\frac{ f_1^6 f_4^5 f_2^2}{f_8^2}\notag\\&\quad+1536\frac{f_1^6 f_4^5 f_2^2}{f_8^2}+512qf_1^2f_4^9+1536qf_1^2f_4^{9} \nonumber \\
        &\equiv-320\frac{f_1^6 f_2^{10} f_4}{f_8^2}-320\frac{f_1^6 f_2^{10} f_4}{f_8^2}+640\frac{ f_2^6 f_4^5}{f_1^2 f_8^2}\nonumber\\
        &\equiv320\frac{ f_2^2 f_4 }{f_1^{10} f_8^2} \left(f_1^8 f_2^4-f_4^4\right)^2\pmod{2048}.\label{finalstepmod2048}
    \end{align}
    Now, it can be easily verified that $f_1^8 f_2^4-f_4^4\equiv0\pmod{8}$. Hence,
    \begin{align}
        f_1^8 f_2^4-f_4^4=8\sum_{n=0}^{\infty}c(n)q^n,\label{binomialmod8}
    \end{align}where $\displaystyle{\sum_{n=0}^{\infty}c(n)q^n}$ is some power series in $q$.
    
    Employing \eqref{binomialmod8} in \eqref{finalstepmod2048}, we arrive at
    \begin{align*}
         \sum_{n=0}^{\infty}c\phi_{16}(8n+7)q^n &\equiv320\times 8^2\frac{ f_2^2 f_4  }{f_1^{10} f_8^2}\left(\sum_{n=0}^{\infty}c(n)q^n\right)^2\equiv0\pmod{2048},
    \end{align*}
    which is \eqref{eq:c_phi_2048}.
    \end{proof}

    \section{Proof of Theorem \ref{thm 18 2 power}}
    The following known 2-, 3-dissections are necessary for our proofs:
    \begin{lemma}\label{Lemma 2d}\cite[Eq. (22.1.12), Lemma 3]{Hir17, matching}
    We have
    \begin{align}
        \frac{1}{f_1f_3}&=\frac{f_{8}^2 f_{12}^5}{f_{2}^2 f_{4} f_{6}^4 f_{24}^2}+q\frac{f_{4}^5 f_{24}^2}{f_{2}^4 f_{6}^2 f_{8}^2 f_{12}},\label{2dis1byf1f3}\\
        \frac{f_3}{f_1}&=\frac{f_{4} f_{6} f_{16} f_{24}^2}{f_{2}^2 f_{8} f_{12} f_{48}}+q\frac{f_{6} f_{8}^2 f_{48}}{f_{2}^2 f_{16} f_{24}}.\label{2disf3byf1}\\
        a(q)&=a(q^4)+6\dfrac{f_4^2f_{12}^2}{f_2f_6},\label{eq:Borwein}
    \end{align}
    where $\displaystyle{a(q):=\sum _{m,n=-\infty }^\infty q^{m^2+mn+n^2}}$ is one of the Borwein cubic functions.
    \end{lemma}
    \begin{lemma}\label{Lemma 3d}
    \cite{Hir17} We have
    \begin{align}
    f_1^3  &=a(q^3) f_3-3q f_9^3,\label{3d f1^3}\\
    \frac{1}{f_1^3}&=a^2(q^3)\frac{f_9^3}{f_3^{10}}+3 a(q^3)) q\frac{f_9^6}{f_3^{11}}+9 q^2\frac{f_9^9}{f_3^{12}},\label{3d 1/f1^3}\\
    \frac{f_1^2}{f_2} &=\frac{f_9^2}{f_{18}}-2 q\frac{f_3 f_{18}^2}{f_6 f_9},\label{3d f1^2/f2}\\
       \frac{f_2}{f_1^2} &=\frac{f_{6}^4 f_{9}^6}{f_{3}^8 f_{18}^3}+2q\frac{f_{6}^3 f_{9}^3}{f_{3}^7}+4q^2\frac{f_{6}^2 f_{18}^3}{f_{3}^6} \label{3d f2/f1^2}\\
       \frac{f_1 f_4}{f_2}&=\frac{f_3 f_{12} f_{18}^5}{f_6^2 f_9^2 f_{36}^2}-q\frac{f_9 f_{36}}{f_{18}},\label{3d f1f4/f2}\\
       \frac{f_2}{f_1 f_4}&=\frac{f_{18}^9}{f_3^2 f_9^3 f_{12}^2 f_{36}^3}+q \frac{f_6^2 f_{18}^3}{f_3^3 f_{12}^3}+q^2\frac{f_6^4 f_9^3 f_{36}^3}{f_3^4 f_{12}^4 f_{18}^3},\label{3d f2/f1f4}\\
       \frac{f_1}{f_4}&=\frac{f_6 f_9 f_{18}}{f_{12}^3}-q \frac{f_3 f_{18}^4}{f_9^2 f_{12}^3}-q^2 \frac{f_6^2 f_9 f_{36}^3}{f_{12}^4 f_{18}^2},\label{3d f1/f4}\\
      \frac{f_4}{f_1} &=\frac{f_{12} f_{18}^4}{f_3^3 f_{36}^2}+q \frac{f_6^2 f_9^3 f_{36}}{f_3^4 f_{18}^2}+2q^2 \frac{f_6 f_{18} f_{36}}{f_3^3},\label{3d f4/f1}\\
        a(q)&=a(q^3)+6q\frac{f_9^3}{f_3}.\label{3d a}
    \end{align}
    \end{lemma}
\begin{lemma}
    \cite[Sect. 22.10 and (22.1.6), (10.9.4)]{Hir17} We have
    \begin{align}
       \label{a(q) Id}a(q)&= \frac{f_1^3}{f_3}+9q\frac{f_9^3}{f_3},\\
       a(q)&=4\frac{f_2^6 f_3}{f_1^3 f_6^2}-\frac{f_1^6 f_6}{f_2^3 f_3^2}.\label{id:a(q)}
    \end{align}
\end{lemma}
We also need the following new identity.
\begin{lemma}\label{Lemma for mod 8}
We have
\begin{align*}
    \frac{f_1^2 f_6}{f_2 f_3^2}+\frac{f_2^2 f_{12}}{f_4 f_6^2}&=2 \frac{f_1 f_{16} f_{24}^2}{f_3 f_6 f_8 f_{48}}.
\end{align*}
\end{lemma}
\begin{proof}
We prove the identity with Maple commands developed by Garvan \cite{Gar19}. The commands use the packages \texttt{$q$-series} and \texttt{ETA}. The instructions for the installation of the packages \texttt{$q$-series} and \texttt{ETA} are given by Garvan \cite{Gar99}. In \cite[Secion 3]{Gar19}, Garvan described how one can prove eta-function identities with Maple commands. The reader is referred to  \cite[Secion 3]{Gar19} for a detailed theory and example. 

The Dedekind eta-function is defined by
\begin{align*}
    \eta({\tau})=q^{1/24}\prod_{n=1}^\infty (1-q^n),
\end{align*}
where $\tau\in\{\tau\in\mathbb{C}: \textup{Im}(\tau)>0\}$ and $q:=e^{2\pi i \tau}$. In terms of eta-function, we need to prove that 
\begin{align*}
    \frac{\eta(\tau)^2 \eta(6\tau)}{\eta(2\tau) \eta(3\tau)^2}+\frac{\eta(2\tau)^2 \eta(12\tau)}{\eta(4\tau) \eta(6\tau)^2}-2 \frac{\eta(\tau) \eta(16\tau) \eta(24\tau)^2}{\eta(3\tau) \eta(6\tau) \eta(8\tau) \eta(48\tau)}&=0.
\end{align*}
The following commands prove the above identity.
\begin{align*}
    &>\,\texttt{with(qseries):}\\
    &>\,\texttt{with(ETA):}\\
    &>\, \texttt{gpP}\coloneqq [1, 2, 2, -1, 3, -2, 6, 1]:\\
    &>\, \texttt{gpQ}\coloneqq [2, 2, 4, -1, 6, -2, 12, 1]:\\
    &>\, \texttt{gpS}\coloneqq [1, 1, 3, -1, 6, -1, 8, -1, 16, 1, 24, 2, 48, -1]:\\
    &>\, \texttt{P}\coloneqq \texttt{gp2etaprod(gpP);}\\
    &\quad\quad \textcolor{blue}{ P\coloneqq \frac{\eta \! \left(\tau \right)^{2} \eta \! \left(6 \tau \right)}{\eta \! \left(2 \tau \right) \eta \! \left(3 \tau \right)^{2}}}\\
    &>\, \texttt{Q}\coloneqq \texttt{gp2etaprod(gpQ);}\\
    &\quad\quad \textcolor{blue}{ Q\coloneqq \frac{\eta \! \left(2 \tau \right)^{2} \eta \! \left(12 \tau \right)}{\eta \! \left(4 \tau \right) \eta \! \left(6 \tau \right)^{2}}}\\
     &>\, \texttt{S}\coloneqq \texttt{gp2etaprod(gpS);}\\
    &\quad\quad \textcolor{blue}{S\coloneqq \frac{\eta \! \left(\tau \right) \eta \! \left(16 \tau \right) \eta \! \left(24 \tau \right)^{2}}{\eta \! \left(3 \tau \right) \eta \! \left(6 \tau \right) \eta \! \left(8 \tau \right) \eta \! \left(48 \tau \right)}}\\
    &>\, \texttt{ETAid}\coloneqq \texttt{P + Q - 2S;}\\
    &\quad\quad \textcolor{blue}{\mathit{ETAid}\coloneqq \frac{\eta \! \left(\tau \right)^{2} \eta \! \left(6 \tau \right)}{\eta \! \left(2 \tau \right) \eta \! \left(3 \tau \right)^{2}}+\frac{\eta \! \left(2 \tau \right)^{2} \eta \! \left(12 \tau \right)}{\eta \! \left(4 \tau \right) \eta \! \left(6 \tau \right)^{2}}-\frac{2 \eta \! \left(\tau \right) \eta \! \left(16 \tau \right) \eta \! \left(24 \tau \right)^{2}}{\eta \! \left(3 \tau \right) \eta \! \left(6 \tau \right) \eta \! \left(8 \tau \right) \eta \! \left(48 \tau \right)}}\\
     &>\, \texttt{ETAidn}\coloneqq \texttt{etanormalid(\%)}\\
    &\quad\quad \textcolor{blue}{\mathit{ETAidn}\coloneqq 1+\frac{\eta \! \left(2 \tau \right)^{3} \eta \! \left(3 \tau \right)^{2} \eta \! \left(12 \tau \right)}{\eta \! \left(\tau \right)^{2} \eta \! \left(6 \tau \right)^{3} \eta \! \left(4 \tau \right)}-\frac{2 \eta \! \left(2 \tau \right) \eta \! \left(3 \tau \right) \eta \! \left(16 \tau \right) \eta \! \left(24 \tau \right)^{2}}{\eta \! \left(\tau \right) \eta \! \left(6 \tau \right)^{2} \eta \! \left(8 \tau \right) \eta \! \left(48 \tau \right)}}\\
    &>\, \texttt{t1}\coloneqq \texttt{op(2, ETAidn)}\\
    &\quad\quad \textcolor{blue}{\mathit{t1}\coloneqq \frac{\eta \! \left(2 \tau \right)^{3} \eta \! \left(3 \tau \right)^{2} \eta \! \left(12 \tau \right)}{\eta \! \left(\tau \right)^{2} \eta \! \left(6 \tau \right)^{3} \eta \! \left(4 \tau \right)}}\\
     &>\, \texttt{t2}\coloneqq \texttt{-op(3, ETAidn)/2}\\
    &\quad\quad \textcolor{blue}{\mathit{t2}\coloneqq \frac{\eta \! \left(2 \tau \right) \eta \! \left(3 \tau \right) \eta \! \left(16 \tau \right) \eta \! \left(24 \tau \right)^{2}}{\eta \! \left(\tau \right) \eta \! \left(6 \tau \right)^{2} \eta \! \left(8 \tau \right) \eta \! \left(48 \tau \right)}}\\
    &>\, \texttt{provemodfuncGAMMA0id(1 + t1 - 2t2, 48);}\\
    &\quad \quad \textcolor{blue}{\textup{TERM ", 1, ``of ", 3, " *****************}}\\
    &\quad \quad \textcolor{blue}{\textup{TERM ", 2, ``of ", 3, " *****************}}\\
    &\quad \quad \textcolor{blue}{\textup{TERM ", 3, ``of ", 3, " *****************}}\\
    &\quad \quad \textcolor{blue}{\textup{``mintotord = ", -2}}\\
    &\quad \quad \textcolor{blue}{\textup{``TO PROVE the identity we need to show that v[oo](ID) $>$ ", 2}}\\
    &\quad \quad \textcolor{blue}{\textup{*** There were NO errors. }}\\
    &\quad \quad \textcolor{blue}{\textup{*** o Each term was modular function on Gamma0(48).}}\\
    &\quad \quad \textcolor{blue}{\textup{*** o We also checked that the total order of each term was zero.}}\\
    &\quad \quad \textcolor{blue}{\textup{*** WARNING: some terms were constants. ***}}\\
    &\quad \quad \textcolor{blue}{\textup{See array CONTERMS.}}\\
    &\quad \quad \textcolor{blue}{\textup{To prove the identity we will need to verify if up to $q^{3}$.}}\\
    &\quad \quad \textcolor{blue}{\textup{Do you want to prove the identity? (yes/no)}}\\
    &\quad \quad \textcolor{blue}{\textup{You entered yes.}}\\
    &\quad \quad \textcolor{blue}{\textup{We verify the identity to $O(q^{98})$.}}\\
    &\quad \quad \textcolor{blue}{\textup{RESULT: The identity holds to $O(q^{98})$.}}\\
     &\quad \quad \textcolor{blue}{\textup{CONCLUSION: This proves the identity since we had only to show that v[oo](ID) $>$ 2.}}
\end{align*}
Thus, we complete the computer-assisted proof.
\end{proof}
\begin{remark}
    As an alternate proof, one may also compute that
    \begin{align*}
        \displaystyle{\frac{\eta \! \left(2 \tau \right)^{3} \eta \! \left(3 \tau \right)^{2} \eta \! \left(12 \tau \right)}{\eta \! \left(\tau \right)^{2} \eta \! \left(4 \tau \right) \eta \! \left(6 \tau \right)^{3}}}
    \end{align*}
    and 
    \begin{align*}
        \displaystyle{ \frac{\eta \! \left(2 \tau \right) \eta \! \left(3 \tau \right) \eta \! \left(16 \tau \right) \eta \! \left(24 \tau \right)^{2}}{\eta \! \left(\tau \right) \eta \! \left(6 \tau \right)^{2} \eta \! \left(8 \tau \right) \eta \! \left(48 \tau \right)}}
    \end{align*}
    are modular forms on $\Gamma_0(48)$ of weight 0 and then prove that
    \begin{align*}
        1+\frac{\eta \! \left(2 \tau \right)^{3} \eta \! \left(3 \tau \right)^{2} \eta \! \left(12 \tau \right)}{\eta \! \left(\tau \right)^{2} \eta \! \left(6 \tau \right)^{3} \eta \! \left(4 \tau \right)}-2\frac{ \eta \! \left(2 \tau \right) \eta \! \left(3 \tau \right) \eta \! \left(16 \tau \right) \eta \! \left(24 \tau \right)^{2}}{\eta \! \left(\tau \right) \eta \! \left(6 \tau \right)^{2} \eta \! \left(8 \tau \right) \eta \! \left(48 \tau \right)}=0,
    \end{align*} 
    which is equivalent to the identity in Lemma \ref{Lemma for mod 8}.. For computation, one may follow \cite[Theorems 1.64 and 1.65]{Ono04} and then use the following lemma
    \begin{lemma}\cite{MDG15}
        Let $f=\displaystyle{\sum_{n=0}^\infty a(n)q^n}$ and $g=\displaystyle{\sum_{n=0}^\infty b(n)q^n}$ be modular forms on some congruence subgroup $\Gamma$ of $SL_2(\mathbb{Z})$ of weight $k$. Suppose that $a(n)=b(n)$ for all $n\le k [SL_2(\mathbb{Z}):\{\pm\}\Gamma]/12$, then $f=g$.
    \end{lemma}
\end{remark}

     Cui, Gu, and Tang \cite[Theorem 3.13]{CGT25} found that
    \begin{align}
        \sum_{n=0}^\infty c\phi_{18}(n)q^n&=\dfrac{1}{f_1^{18}}\bigg (a^6(q)a(q^6) \dfrac{f_6^{15}}{f_3^6f_{12}^6}+270qa^4(q)a(q^6) \dfrac{f_2^2f_3^3f_6^9}{f_1^3f_4f_{12}^3} +7290q^2a^2(q)\dfrac{f_3^6f_6^{18}}{f_1^4f_2f_{12}^6}\nonumber \\&\quad +1080q^2a^3(q)\dfrac{f_3^8f_4f_6^6}{f_1^2f_2f_{12}} +1080q^2a^3(q)\dfrac{f_3^8f_4f_6^6f_9^2}{f_1^4f_{12}f_{18}} +1080q^2a^3(q)\dfrac{f_2f_3^6f_6^{12}}{f_1^4f_4f_{12}^3}\nonumber \\&\quad +14580q^3\dfrac{f_2f_3^{15}f_6^{12}}{f_1^7f_4f_{12}^3} +58320q^3a(q)\dfrac{f_3^{14}f_4^3f_6^2f_9^2f_{12}}{f_1^6f_{18}}\nonumber \\&\quad +29160q^3a^2(q)\dfrac{f_2^3f_3^{12}f_6^2f_{12}^2}{f_1^6} +2160q^3a^3(q)\dfrac{f_2^3f_3^{15}f_{12}^2f_{18}^2}{f_1^6f_6^4f_9}\nonumber \\&\quad +2160q^3a^4(q)\dfrac{f_3^6f_4^3f_6f_{12}f_{18}^2}{f_1^3f_9} +24q^3a^6(q)\dfrac{f_6^5f_{12}^2f_{18}^2}{f_3^3f_9}+29160q^4\dfrac{f_2^2f_3^{20}f_{12}^5}{f_1^6f_4f_6^3}\nonumber \\&\quad +13122q^4\dfrac{f_3^{16}f_9^4f_{18}^5}{f_1^6f_{36}^2} +58320q^4\dfrac{f_3^{18}f_4f_6^5f_{18}^2}{f_1^7f_9f_{12}} +29160q^4\dfrac{f_2^3f_3^{20}f_9^2f_{12}^5}{f_1^8f_4f_6^3f_{18}}\nonumber \\&\quad +43740q^4a(q)\dfrac{f_2^2f_3^{12}f_6^8f_{18}^3}{f_1^6f_4f_{12}^3} +8748q^4a(q^3)\dfrac{f_3^{16}f_6^2f_9^4f_{36}}{f_1^6f_{12}f_{18}}\bigg )\label{c 18 Gen}.
    \end{align}
    Working modulo 8 on both sides of the above, we have
    \begin{align*}
         \sum_{n=0}^\infty c\phi_{18}(n)q^n&\equiv a(q)q(q^6)\frac{f_6^{15}}{f_1^{18} f_3^6 f_{12}^6}+2\times 5\times 3^3a(q^6)q \frac{f_2^2 f_3^3 f_6^9}{f_1^{21} f_4 f_{12}^3}+2\times 5\times 3^6 q^2\frac{f_3^6 f_6^{18}}{f_1^{22} f_2 f_{12}^6}\\ &\quad +5\times 2^2\times 3^7q^4\frac{f_2^2 f_3^{12} f_6^8 f_{18}^3}{f_1^{24} f_4 f_{12}^3}+5\times 2^2\times 3^6q^3\frac{f_2 f_3^{15} f_6^{12}}{f_1^{25} f_4 f_{12}^3}+2^2 \times 3^7 q^4\frac{f_3^{16} f_6^2 f_9^4 f_{36}}{f_1^{24} f_{12} f_{18}}\\&\quad+2\times 3^8q^4\frac{f_3^{16} f_9^4 f_{18}^5}{f_1^{24} f_{36}^2}\pmod8.
    \end{align*}
    Applying \eqref{eq:Borwein} in the above identity, we obtain
    \begin{align*}
       \sum_{n=0}^\infty c\phi_{18}(n)q^n&\equiv a^6(q^4) a(q^6)\frac{f_6^{15}}{f_1^{18} f_3^6 f_{12}^6}+4 a^5(q^4) a(q^6) q \frac{f_4^2 f_6^{14}}{f_1^{18} f_2 f_3^6 f_{12}^4}-2 a(q^6) q\frac{f_2^2 f_3^3 f_6^9}{f_1^{21} f_4 f_{12}^3}\\&\quad+4 a^4(q^4) a(q^6) q^2\frac{f_4^4 f_6^{13}}{f_1^{18} f_2^2 f_3^6 f_{12}^2}+2 q^2\frac{f_3^6 f_6^{18}}{f_1^{22} f_2 f_{12}^6}+4q^3\frac{f_2 f_3^{15} f_6^{12}}{f_1^{25} f_4 f_{12}^3}\\&\quad+4q^4\frac{f_2^2 f_3^{12} f_6^8 f_{18}^3}{f_1^{24} f_4 f_{12}^3}+2q^4\frac{f_3^{16} f_9^4 f_{18}^5}{f_1^{24} f_{36}^2}+4q^4\frac{f_3^{16} f_9^4 f_6^2 f_{36}}{f_1^{24} f_{12} f_{18}}\\
       &\equiv a^6(q^4) a(q^6)\frac{f_6^{15}}{f_1^{18} f_3^6 f_{12}^6}+4 a^5(q^4) a(q^6) q \frac{f_4^2 f_6^{14}}{f_1^{18} f_2 f_3^6 f_{12}^4}-2 a(q^6) q\frac{f_2^2 f_3^3 f_6^9}{f_1^{21} f_4 f_{12}^3}\\&\quad+4 a^4(q^4) a(q^6) q^2\frac{f_4^4 f_6^{13}}{f_1^{18} f_2^2 f_3^6 f_{12}^2}+2 q^2\frac{f_3^6 f_6^{18}}{f_1^{22} f_2 f_{12}^6}+4q^3\frac{f_2 f_3^{15} f_6^{12}}{f_1^{25} f_4 f_{12}^3}\\&\quad+2q^4\frac{f_3^{16} f_9^4 f_{18}^5}{f_1^{24} f_{36}}\pmod8.
    \end{align*}
    Using \eqref{dis1byf1^2}, \eqref{dis1byf1^4}, \eqref{2dis1byf1f3}, and \eqref{2disf3byf1} in the above identity, extracting the odd powered terms of $q$, and then simplifying, we find that
    \begin{align*}
         \sum_{n=0}^\infty c\phi_{18}(2n+1)q^n&\equiv 2a(q^3)\frac{f_4^2 f_6^2 f_3^7}{f_1^{10} f_2^2 f_{12}^2}+2 a^6(q^2) a(q^3) \frac{f_2^2 f_6^8 f_8^2 f_{12}}{f_1^{13} f_3^4 f_4 f_{24}^2}+4 a^5(q^2) a(q^3) \frac{f_2^2 f_3^{11}}{f_1^{10} f_6^4}\\&\quad-2a^6(q^2) a(q^3) q\frac{f_4^5 f_6^{10} f_{24}^2}{f_1^{13} f_3^4 f_8^2 f_{12}^5}+4 a^6(q^2) a(q^3)q\frac{f_4^5 f_{12}^9}{f_1^{13} f_6^4 f_8^2 f_{24}^2}\\&\quad+4q\frac{f_2^2 f_8^2 f_3^{21} f_{12}^5}{f_1^{16} f_4 f_6^8 f_{24}^2}+4q\frac{f_8  f_3^{20} f_{12}^2}{f_1^{13} f_4 f_6^4 f_{24}}-4q^2\frac{f_4^5 f_3^{21} f_{24}^2}{f_1^{16} f_6^6 f_8^2 f_{12}}\pmod8 .
    \end{align*}
    Invoking \eqref{a(q) Id}, \eqref{3d f1^3}, \eqref{3d 1/f1^3}, \eqref{3d f1^2/f2}, and \eqref{3d f2/f1^2} in the above, extracting the terms involving $q^{3n+2}$, replacing $q^3$ by $q$ and then simplifying, we have
      \begin{align}
         \sum_{n=0}^\infty c\phi_{18}(6n+5)q^n&\equiv2 a^6(q) a^6(q^2) \frac{f_4 f_3^{16}}{f_1^{23} f_{12}}+2 a^3(q) \frac{f_3^{12}}{f_1^{19}}-4 a^8(q) \frac{f_4^3 f_3^{12}}{f_1^{31}}+4 a^6(q) a^6(q^2) \frac{f_3^{12}}{f_1^{19}}\nonumber\\&\quad+4 a^5(q) \frac{f_3^{12}}{f_1^{19}}+4 a^3(q) a^5(q^2) \frac{f_3^{12}}{f_1^{19}}+4 a^2(q) a^6(q^2)\frac{f_6^3}{f_1}+4 a^3(q) a^6(q^2) q\frac{f_3^{21}}{f_1^{22}}\nonumber\\&\quad+4 a^2(q) a^6(q^2) q^2\frac{f_3^{30}}{f_1^{25}}\pmod8 .\label{eq:6n+5}
    \end{align}
        From \eqref{id:a(q)}, it can be easily seen that
    \begin{align*}
        a(q)&\equiv1 \pmod2,  &
        a(q)&\equiv\dfrac{f_1^6f_6}{f_2^3f_3^2}\pmod4, & 
        a^2(q)&\equiv1\pmod4.
    \end{align*}
    Invoking the above in \eqref{eq:6n+5}, we obtain
    \begin{align*}
         \sum_{n=0}^\infty c\phi_{18}(6n+5)q^n
         &\equiv 2 \frac{f_4 f_3^{16}}{f_1^{23} f_{12}}+2 \frac{f_2f_3^{10}f_6}{f_1^{21}}+4\frac{f_6^3}{f_1}+4q \frac{f_3^{21}}{f_1^{22}}\left(1+q\frac{f_3^9}{f_1^3}\right)\pmod8\\
         &\equiv 2 \frac{f_4 f_3^{16}}{f_1^{23} f_{12}}+2 \frac{f_2f_3^{10}f_6}{f_1^{21}}+4\frac{f_6^3}{f_1}+4q\frac{f_1^3f_3^{21}}{f_1^{25}}\left(\frac{f_1^9}{f_3^3}\right) \quad\text{(Thanks to \cite[(21.3.2)]{Hir17})}\\
         &\equiv 2 \frac{f_4 f_3^{16}}{f_1^{23} f_{12}}+2 \frac{f_2f_3^{10}f_6}{f_1^{21}}+4\frac{f_3^6}{f_1}\left(1+q\frac{f_3^{12}}{f_1^{12}}\right)\\
         &\equiv 2 \frac{f_4 f_3^{16}}{f_1^{23} f_{12}}+2 \frac{f_2f_3^{10}f_6}{f_1^{21}}+4\frac{f_3^6}{f_1}\left(\frac{f_3^3}{f_1^9}\right) \quad\text{(Thanks to \cite[Sect. 22.10]{Hir17})}\\
         &\equiv 2 \frac{f_4 f_3^{16}}{f_1^{23} f_{12}}+2 \frac{f_2f_3^{10}f_6}{f_1^{21}}+4\frac{f_3^9}{f_1^{10}}\\
         &\equiv 2 \frac{f_3^{12}}{f_1^{19}}\left(\frac{f_4 f_{12}^3}{f_2^2 f_3^{12}}+\frac{f_1^2 f_6}{f_2 f_3^2}\right)+4\frac{f_3^9}{f_1^{10}}\\
         &\equiv 2 \frac{f_3^{12} }{f_1^{19}}\left(\frac{f_1^2 f_6 }{f_2 f_3^2}+\frac{f_2^2 f_{12}}{f_4 f_6^2}\right)+4\frac{f_3^9}{f_1^{10}}\\
         &\equiv 2 \frac{f_3^{12}}{f_1^{19}}\left(2\frac{f_1 f_{16} f_{24}^2}{f_3 f_6 f_8 f_{48}}\right)+4\frac{f_3^9}{f_1^{10}}\quad\text{(Thanks to Lemma \ref{Lemma for mod 8})}\\
         &\equiv 4 \frac{f_3^{11} f_{16} f_{24}^2}{f_1^{18} f_6 f_8 f_{48}}+4\frac{f_3^9}{f_1^{10}}\\
         &\equiv 4\frac{f_3^9}{f_1^{10}}+4\frac{f_3^9}{f_1^{10}}\\
         &\equiv 0\pmod8,
    \end{align*}
    which is \eqref{cong: mod 8}. This completes the proof. \qed
    \section{Proof of Theorem \ref{thm 18 3 power}}
    \begin{proof}[Proof of \eqref{c 18 81}]
    From \eqref{a(q) Id}, we recall that $a(q)\equiv 1 \pmod{3}$. Now, working modulo 81, \eqref{c 18 Gen} becomes
    \begin{align*}
     \sum_{n=0}^\infty c\phi_{18}(n)q^n&\equiv \frac{1}{f_1^{18}}\Bigg( a^6(q) a(q^6)\frac{f_6^{15}}{f_3^6 f_{12}^6}+2\times 3^3 \times5 q \frac{f_2^2 f_3^3 f_6^9}{f_1^3 f_4 f_{12}^3}+2^3\times 3^3\times 5 q^2\frac{f_3^8 f_4 f_6^6}{f_1^2 f_2 f_{12}}\\
     &\quad +2^3\times 3^3\times 5 q^2\frac{f_3^8 f_4 f_6^6 f_9^2}{f_1^4 f_{12} f_{18}}+2^3\times 3^3\times 5 q^2\frac{f_2 f_3^6 f_6^{12}}{f_1^4 f_4 f_{12}^3}+2^4\times 3^3\times 5 q^3\frac{f_2^3 f_3^{15} f_{12}^2 f_{18}^2}{f_1^6 f_6^4 f_9}\\
     &\quad +2^4\times 3^3\times 5 q^3\frac{f_3^6 f_4^3 f_6 f_{12} f_{18}^2}{f_1^3 f_9}+2^3 \times3 q^3\frac{f_1^{18} f_6^5 f_{12}^2 f_{18}^2}{f_3^9 f_9}\Bigg) \pmod{81}.
    \end{align*}
    Now, we employ Lemma \ref{Lemma 3d} to 3-dissect the above identity in order to obtain
    \begin{align*}
        \sum_{n=0}^\infty c\phi_{18}(3n+1)q^n&\equiv 2\times 3^3 \frac{f_2^{15} f_3^{21}}{f_1^{67} f_4^6}+2 \times3^3\times 5\frac{f_3^3 f_2^9 f_3^{18} f_2^6}{f_1^{67} f_4^3 f_4^3}+2^4\times 3^3\times 5 q\frac{f_2^7 f_3^{20} f_{12}}{f_1^{56} f_4}\\
        &\quad+2^3\times 3^3\times 5 q\frac{f_2^{16} f_3^{21} f_{12}^3}{f_1^{59} f_4^7 f_6^3}\\
        &\equiv 324 \frac{f_2^{15} f_3^{21}}{f_1^{67} f_4^6}+ 3240 q\frac{f_2^7 f_3^{20} f_{12}}{f_1^{56} f_4}\equiv 0 \pmod{81},
    \end{align*}which proves \eqref{c 18 81}.
    \end{proof}

\begin{proof}[Proof of \eqref{c 18 2187}]
     We now turn to the proof of \eqref{c 18 2187}, the conjecture proposed by Cui, Gu, and Tang \cite{CGT25}. We first reduce the generating function modulo 2187 and then complete the proof using the $(p,k)$-parametrization.
    Our approach requires the following theta function identities.
    We recall the following $(p,k)$-parametrizations from \cite{AAW06}, which will be used throughout the proof:
    \begin{align*}
        f_1&=2^{-1/6} q^{-1/24} p^{1/24} (1 - p)^{1/2} (1+p)^{1/6} (1+2p )^{1/8} (2+p)^{1/8} k^{1/2},\\
        f_2&=2^{-1/3} \, q^{-1/12} \, p^{1/12} (1 - p)^{1/4} (1+p)^{1/12} (1+2p)^{1/4} (2+p)^{1/4} k^{1/2},\\
        f_3&=2^{-1/6} \, q^{-1/8} \, p^{1/8} (1 - p)^{1/6} (1+p)^{1/2} (1 + 2p)^{1/24} (2 + p)^{1/24} k^{1/2},\\
        f_4&=2^{-2/3} \, q^{-1/6} \, p^{1/6} (1 - p)^{1/8} (1+p)^{1/24} (1+2p )^{1/8} (2+p)^{1/2} k^{1/2},\\
        f_6&=2^{-1/3} \, q^{-1/4} \, p^{1/4} (1 - p)^{1/12} (1+p)^{1/4} (1+2p )^{1/12} (2+p)^{1/12} k^{1/2},\\
        f_{12}&=2^{-2/3} \, q^{-1/2} \, p^{1/2} (1 - p)^{1/24} (1+p)^{1/8} (1+2p)^{1/24} (2+p)^{1/6} k^{1/2},
    \end{align*}
    where
    \begin{align*}
        p:=p(q)=\frac{\Theta_3^2(q)-\Theta_3^2(q^3)}{2\Theta_3^2(q^3)}, \qquad
        k:=k(q)=\frac{\Theta_3^3(q^3)}{\Theta_3(q)}.
    \end{align*}
    \begin{lemma}
        We have
        \begin{align}
            f_1^4f_6^2+f_2^2 f_3^4&=2\frac{f_1^2 f_4 f_6^9}{f_2 f_3^2 f_{12}^3},\label{Id 1}\\
            f_2^2 f_3^4- f_1^4 f_6^2&= 4q\dfrac{f_1f_2^2f_3f_{12}^3}{f_4},\label{Id 4}\\
            f_2^9f_3 f_{12}^2+2 f_1^3 f_4^6 f_6^3&=3 f_1^2 f_2^2 f_3^3 f_4^3 f_6 f_{12},\label{Id 2}\\
            f_2^9 f_3 f_{12}^2-f_1^3 f_4^6 f_6^3&=3q f_1^3 f_2^2 f_4^2 f_6 f_{12}^4, \label{Id 5}\\
            f_2^2f_3^3f_{12}-f_1f_4^3f_6^2&=q \dfrac{f_1f_2^2f_{12}^4}{f_4}, \label{Id 6}\\
            2 f_1^2 f_3^2 f_4^3 f_6-f_2^7 f_{12}&=\frac{f_1^4 f_4^4 f_6^{10}}{f_2^3 f_3^4 f_{12}^3}, \label{Id 8}\\
            4 f_3 f_4^3-f_1^3 f_{12}&=3\frac{f_2^2 f_3^4 f_{12}}{f_1 f_6^2}. \label{Id 9}
        \end{align}
    \end{lemma}
    \begin{proof}
        Using these parametrizations,
    In view of the above parametrization, we have
    \begin{align*}
        f_1^4f_6^2+f_2^2 f_3^4&=2^{-4/3} \, q^{-2/3} \, p^{2/3} (1 - p)^{13/6} (1+p)^{7/6} (1+2p)^{2/3} (2+p)^{2/3} k^3\\
        &\quad+ 2^{-4/3} \, q^{-2/3} \, p^{2/3} (1 - p)^{7/6} (1+p)^{13/6} (1+2p)^{2/3} (2+p)^{2/3} k^3.
    \end{align*}
    On simplification, we have
    \begin{align*}
        f_1^4f_6^2+f_2^2 f_3^4&=2^{-1/3} \, q^{-2/3} \, p^{2/3} (1 - p^2)^{7/6} (1+2p)^{2/3} (2+p)^{2/3} k^3=2 \frac{f_1^2 f_4 f_6^9}{f_2 f_3^2 f_{12}^3},
    \end{align*}
    which is \eqref{Id 1}. The remaining identities can similarly be proved.
    \end{proof}
    
    Now, first of all, we replace all the $a(q)$ in \eqref{c 18 Gen}  suitably by
    \begin{align*}
        a(q)&=\frac{f_1^3}{f_3}+9q \frac{f_9^3}{f_3}, &
        a^3(q)&=\frac{f_1^9}{f_3^3}+27 q \frac{f_3^9}{f_1^3}, &
       a(q^3) &=\frac{f_1^3}{f_3}+3q \frac{f_9^3}{f_3}.
    \end{align*}
    Then on the resulting identity, we employ Lemma \ref{Lemma 3d} under modulo 2187 to extract
        \begin{align*}
        &\sum_{n=0}^\infty c\phi_{18}(3n+2)q^n\\&\equiv 1215\frac{f_2^{12} f_3^7 f_6}{f_1^{17} f_4^3 f_{12}}+2106\frac{f_2^{18} f_3^8}{f_1^{20} f_4^6 f_6}+1080 \frac{f_2^{12} f_3 f_6^9}{f_1^{15} f_4^5 f_{12}^3}+1080 \frac{f_2^4 f_3^2 f_6^3}{f_1^{14} f_{12}^2}\left( f_1^4f_6^2+f_2^2 f_3^4\right)\\
        &\quad +729 q\frac{f_2^4 f_3^{14} f_6^3}{f_1^{26} f_{12}^2}\left( f_1^4f_6^2+f_2^2 f_3^4\right)+1215 q\frac{f_3^{11} f_6^6}{f_1^{17} f_2^5 f_{12}^2}\left( f_1^4f_6^2+f_2^2 f_3^4\right)+ 972 q\frac{f_2^6 f_3^{13} f_{12}}{f_1^{23} f_4 f_6^3}\\
        &\quad \times \left( f_2^2 f_3^4-f_1^4f_6^2\right)+729 q \frac{f_3^{12} f_4^2 f_6}{f_1^{16}}+243 q \frac{f_2^5 f_3^7 f_4^2 f_6^2}{f_1^{17}}-729q \frac{f_2 f_3^9 f_4^2 f_6^2}{f_1^{15}}+972q \frac{f_2^{14} f_3^{12} f_6^3}{f_1^{24} f_4^6}\\
        &\quad -2106q \frac{f_3^4 f_4^2 f_6^5}{f_1^8 f_2^4}+1458q \frac{f_3^{13} f_6^6}{f_1^{11} f_2^3 f_4^6}+729 q \frac{f_2^{12} f_3^{13} f_6^9}{f_1^{27} f_4^5 f_{12}^3}+1215 q \frac{f_2^{12} f_3^{10} f_6^9}{f_1^{18} f_4^{14}}+405 q \frac{f_2^7 f_3^{10} f_{12}}{f_1^{18} f_4}\\
        &\quad +2025 q\frac{f_2^{16} f_3^{11} f_{12}^3}{f_1^{21} f_4^7 f_6^3}-1458 q^2\frac{f_3^{13} f_6^5 f_{12}^3}{f_1^{11} f_2^4 f_4^7}+1458 q^2\frac{f_3^8 f_6^{13}}{f_1^{20} f_{12}^2}-729 q^2\frac{f_2 f_3^6 f_4 f_6^2 f_{12}^3}{f_1^{14}}\\
        &\quad +729 q^2\frac{f_3^{19} f_6^3}{f_1^{24} f_2^2 f_4}\left(f_3 f_4^3+2 f_1^3 f_{12}\right) \pmod{2187}.
    \end{align*}

    In pursuit of $\displaystyle{\sum_{n=0}^\infty c\phi_{18}(3n+2)q^n}\equiv 0 \pmod{2187}$, we attempt to reduce the above generating function modulo 2187 in several steps as described below. Due to \eqref{Id 1}, \eqref{Id 4}, and the fact that $f_3 f_4^3\equiv f_1^3 f_{12}\pmod{3}$, the above generating function reduces to
    \begin{align}
        &\sum_{n=0}^\infty c\phi_{18}(3n+2)q^n\notag\\
        &\equiv 3^3\times 2^3\times 5 \frac{f_2^3 f_6^9}{f_1^{15} f_4^5 f_{12}^5} \left(f_2^9 f_3 f_{12}^2+2 f_1^3 f_4^6 f_6^3\right)-3^4 \frac{f_2^{18} f_3^8}{f_1^{20} f_4^6 f_6}+3^4 q\frac{f_3^4 f_4^2 f_6^5}{f_1^8 f_2^4}+3^4\times 5 q\frac{f_2^7 f_3^{10} f_{12}}{f_1^{18} f_4}\notag\\
        &\quad +3^4\times 5^2 q\frac{f_2^{16} f_3^{11} f_{12}^3}{f_1^{21} f_4^7 f_6^3}-2^2\times 3^5\frac{f_2^{12} f_3^7 f_6}{f_1^{17} f_4^3 f_{12}}+2^2\times 3^5 q\frac{f_2^{14} f_3^{12} f_6^3}{f_1^{24} f_4^6}-2^2\times 3^5 q\frac{f_2^{12} f_3^{10} f_6^9}{f_1^{18} f_4^{14}}\notag\\
        &\quad +3^5 q\frac{f_2^5 f_3^7 f_4^2 f_6^2}{f_1^{17}}+3^5 q\frac{f_3^9 f_4 f_6^{15}}{f_1^{15} f_2^6 f_{12}^5}+3^5\times 7 q^2\frac{f_2^8 f_3^{14} f_{12}^4}{f_1^{22} f_4^2 f_6^3}+3^6 q \frac{f_3^9 f_4^2 f_6}{f_1^{16}}\left(f_3^3-f_1 f_2 f_6\right)\notag\\
        &\quad +3^6q \frac{f_2^3 f_3^{12} f_6^9}{f_1^{27} f_4^5 f_{12}^5}\left(f_2^9 f_3 f_{12}^2-f_1^3 f_4^6 f_6^3\right)-3^6 q\frac{f_3^{13} f_6^6}{f_1^{11} f_2^3 f_4^6}-3^6 q^2\frac{f_2 f_3^6 f_4 f_6^2 f_{12}^3}{f_1^{14}}-3^6 q^2\frac{f_3^8 f_6^{13}}{f_1^{20} f_{12}^2}\notag\\
        &\quad +3^6 q^2 \frac{f_3^{13} f_6^5 f_{12}^3}{f_1^{11} f_2^4 f_4^7} \pmod{2187}.\label{Gen 3n+2 Mod 2187 Eve1}
    \end{align}
    Thanks to \eqref{Id 2}, \eqref{Id 5} and \cite{Hir17}
    \begin{align*}
        \frac{f_3^3}{f_1}&=\frac{f_4^3 f_6^2}{f_2^2 f_{12}}+q \frac{f_{12}^3}{f_4} \quad \Rightarrow \quad f_3^3-f_1 f_2 f_6\equiv q \frac{f_1f_{12}^3}{f_4} \pmod{3},
    \end{align*}
    for which, \eqref{Gen 3n+2 Mod 2187 Eve1} reduces to
    \begin{align}
        &\sum_{n=0}^\infty c\phi_{18}(3n+2)q^n\notag\\
        &\equiv 3^4\times 13\frac{f_2^5 f_3^3 f_6^{10}}{f_1^{13} f_4^2 f_{12}^4}-3^4 \frac{f_2^{18} f_3^8}{f_1^{20} f_4^6 f_6}+3^4 q \frac{f_3^4 f_4^2 f_6^5}{f_1^8 f_2^4}+3^4\times 5 q \frac{f_2^7 f_3^{10} f_{12}}{f_1^{18} f_4}+3^4\times  5^2 q \frac{f_2^{16} f_3^{11} f_{12}^3}{f_1^{21} f_4^7 f_6^3}\notag\\
        &\quad -2^2\times 3^5\frac{f_2^{12} f_3^7 f_6}{f_1^{17} f_4^3 f_{12}}+2^2\times 3^5 q\frac{f_2^{14} f_3^{12} f_6^3}{f_1^{24} f_4^6}-2^2\times 3^5 q\frac{f_2^{12} f_3^{10} f_6^9}{f_1^{18} f_4^{14}}+3^5 q\frac{f_2^5 f_3^7 f_4^2 f_6^2}{f_1^{17}}\notag\\
        &\quad +3^5 q\frac{f_3^9 f_4 f_6^{15}}{f_1^{15} f_2^6 f_{12}^5}+3^5\times 7 q^2\frac{f_2^8 f_3^{14} f_{12}^4}{f_1^{22} f_4^2 f_6^3}+3^6 q^2\frac{f_3^6 f_4 f_6 f_{12}^3}{f_1^{15}}(f_3^3-f_1 f_2 f_6)-3^6 q\frac{f_3^{13} f_6^6}{f_1^{11} f_2^3 f_4^6}\notag\\
        &\quad -3^6 q^2\frac{f_3^8 f_6^{13}}{f_1^{20} f_{12}^2}+3^6 q^2\frac{f_3^{13} f_6^5 f_{12}^3}{f_1^{11} f_2^4 f_4^7}\notag\\
        &\equiv 3^4\times 13\frac{f_2^5 f_3^3 f_6^{10}}{f_1^{13} f_4^2 f_{12}^4}-3^4 \frac{f_2^{18} f_3^8}{f_1^{20} f_4^6 f_6}+3^4 q \frac{f_3^4 f_4^2 f_6^5}{f_1^8 f_2^4}+ 3^4\times 5 q \frac{f_2^7 f_3^{10} f_{12}}{f_1^{21} f_4^7 f_6^3}\left(f_1^3 f_4^6 f_6^3+5f_2^9 f_3 f_{12}^2\right)\notag\\
        &\quad -2^2\times 3^5\frac{f_2^{12} f_3^7 f_6}{f_1^{17} f_4^3 f_{12}}+2^2\times 3^5 q\frac{f_2^{14} f_3^{12} f_6^3}{f_1^{24} f_4^6}-2^2\times 3^5 q\frac{f_2^{12} f_3^{10} f_6^9}{f_1^{18} f_4^{14}}+3^5 q\frac{f_2^5 f_3^7 f_4^2 f_6^2}{f_1^{17}}\notag\\
        &\quad +3^5 q\frac{f_3^9 f_4 f_6^{15}}{f_1^{15} f_2^6 f_{12}^5}+3^5\times 7 q^2\frac{f_2^8 f_3^{14} f_{12}^4}{f_1^{22} f_4^2 f_6^3}-3^6 q\frac{f_3^{13} f_6^5}{f_1^{11} f_2^4 f_4^7}(f_2 f_4 f_6-q f_{12}^3)-3^6 q^2\frac{f_3^8 f_6^{13}}{f_1^{20} f_{12}^2}\notag\\
        &\quad +3^6 q^3 \frac{f_3^6 f_6 f_{12}^6}{f_1^{14}} \pmod{2187}. \label{Gen 3n+2 Mod 2187 Eve2}
    \end{align}

Now using \eqref{Id 2},
\begin{align*}
    5\left(f_1^3 f_4^6 f_6^3+5f_2^9 f_3 f_{12}^2\right)&\equiv 2 (2(2f_1^3 f_4^6 f_6^3+f_2^9 f_3 f_{12}^2+6 f_1^3 f_4^6 f_6^3)-3f_2^9 f_3 f_{12}^2)\\
    &\equiv 2 \times 3(2f_1^2 f_2^2 f_3^3 f_4^3 f_6 f_{12}+4 f_1^3 f_4^6 f_6^3-f_2^9 f_3 f_{12}^2) \pmod{27}.
\end{align*}
From \cite{Hir17}
    \begin{align*}
        \frac{f_1}{f_3^3}&=\frac{f_2 f_4^2 f_{12}^2}{f_6^7}-q \frac{f_2^3 f_{12}^6}{f_4^2 f_6^9} \quad \Rightarrow \quad f_2 f_4 f_6-q f_{12}^3\equiv \dfrac{f_1f_4^2f_6^9}{f_2^3f_3^3f_{12}^
        3} \pmod{3}.
    \end{align*}
Applying the above two in \eqref{Gen 3n+2 Mod 2187 Eve2}, we get
    \begin{align}
        &\sum_{n=0}^\infty c\phi_{18}(3n+2)q^n\notag\\
        &\equiv 3^4\times 13\frac{f_2^5 f_3^3 f_6^{10}}{f_1^{13} f_4^2 f_{12}^4}-3^4 \frac{f_2^{18} f_3^8}{f_1^{20} f_4^6 f_6}+3^4 q \frac{f_3^4 f_4^2 f_6^5}{f_1^8 f_2^4}+3^5 q \frac{f_2^5 f_3^7}{f_1^{18} f_4}(f_1 f_4^3 f_6^2-f_2^2 f_3^3 f_{12})\notag\\
        &\quad -2^2\times  3^5\frac{f_2^{12} f_3^7 f_6}{f_1^{17} f_4^3 f_{12}}+ 2 \times 3^5 q \frac{f_2^9 f_3^{11} f_{12}^2}{f_1^{21} f_4^7 f_6^3}(2 f_1^2 f_3^2 f_4^3 f_6-f_2^7 f_{12})+2^2\times 3^5 q \frac{f_2^{14} f_3^{12} f_6^3}{f_1^{24} f_4^6}\notag\\
        &\quad+3^5 q\frac{f_3^9 f_6^9}{f_1^{18} f_2^6 f_4^{14} f_{12}^5}(f_1^3 f_4^{15} f_6^6-4 f_2^{18} f_3 f_{12}^5)+3^5 \times 7 q^2 \frac{f_2^8 f_3^{14} f_{12}^4}{f_1^{22} f_4^2 f_6^3}-3^6 q  \frac{f_3^{10} f_6^{14}}{f_1^{10} f_2^7 f_4^5 f_{12}^3}\notag\\
        &\quad -3^6 q^2\frac{f_3^8 f_6^{13}}{f_1^{20} f_{12}^2}+3^6 q^3 \frac{f_3^6 f_6 f_{12}^6}{f_1^{14}} \pmod{2187}. \label{Gen 3n+2 Mod 2187 Eve3}
    \end{align}
    Here, 
    \begin{align*}
        f_1^3 f_4^{15} f_6^6-4 f_2^{18} f_3 f_{12}^5&\equiv \frac{f_6^6 f_{12}^5}{f_4^3}(f_1^3 f_{12}-4 f_3 f_4^3)\pmod{9}.
    \end{align*}
    Using the above identity, \eqref{Id 6}, \eqref{Id 8}, and \eqref{Id 9} in \eqref{Gen 3n+2 Mod 2187 Eve3} and the fact that $f_2^6f_{12}^2 \equiv f_4^6 f_6^2 \pmod{3}$, we derive that
    \begin{align}
        &\sum_{n=0}^\infty c\phi_{18}(3n+2)q^n\notag\\
        &\equiv 3^4\times 13\frac{f_2^5 f_3^3 f_6^{10}}{f_1^{13} f_4^2 f_{12}^4}-3^4 \frac{f_2^{18} f_3^8}{f_1^{20} f_4^6 f_6}+3^4 q \frac{f_3^4 f_4^2 f_6^5}{f_1^8 f_2^4}-3^5 q^2\frac{f_2^7 f_3^7 f_{12}^4}{f_1^{17} f_4^2}-2^2\times 3^5 \frac{f_2^{12} f_3^7 f_6}{f_1^{17} f_4^3 f_{12}}\notag\\
        &\quad+2\times 3^5 q\frac{f_2^6 f_3^7 f_6^7}{f_1^{17} f_4^3 f_{12}}+2^2 \times 3^5 q\frac{f_2^{14} f_3^{12} f_6^3}{f_1^{24} f_4^6}+3^5\times  7 q^2\frac{f_2^8 f_3^{14} f_{12}^4}{f_1^{22} f_4^2 f_6^3}-3^6q \frac{f_3^{10} f_6^{12}}{f_1^{10} f_2^7 f_4^{11} f_{12}^3}\notag\\
        &\quad\times (f_2^6f_{12}^2 +f_4^6 f_6^2)-3^6 q^2\frac{f_3^8 f_6^{13}}{f_1^{20} f_{12}^2}+3^6 q^3 \frac{f_3^6 f_6 f_{12}^6}{f_1^{14}} \notag\\
        &\equiv 3^4\times 13\frac{f_2^5 f_3^3 f_6^{10}}{f_1^{13} f_4^2 f_{12}^4}-3^4 \frac{f_2^{18} f_3^8}{f_1^{20} f_4^6 f_6}+3^4 q \frac{f_3^4 f_4^2 f_6^5}{f_1^8 f_2^4}-3^5 q^2\frac{f_2^7 f_3^7 f_{12}^4}{f_1^{17} f_4^2}-2^2\times 3^5 \frac{f_2^{12} f_3^7 f_6}{f_1^{17} f_4^3 f_{12}}\notag\\
        &\quad+2\times 3^5 q\frac{f_2^6 f_3^7 f_6^7}{f_1^{17} f_4^3 f_{12}}+2^2 \times 3^5 q\frac{f_2^{14} f_3^{12} f_6^3}{f_1^{24} f_4^6}+3^5\times  7 q^2\frac{f_2^8 f_3^{14} f_{12}^4}{f_1^{22} f_4^2 f_6^3}+3^6q \frac{f_3^{10} f_6^{12}}{f_1^{10} f_2 f_4^{11} f_{12}}\notag\\
        &\quad-3^6 q^2\frac{f_3^8 f_6^{13}}{f_1^{20} f_{12}^2}+3^6 q^3 \frac{f_3^6 f_6 f_{12}^6}{f_1^{14}} \pmod{2187}.\label{Gen 3n+2 Mod 2187 Eve4}
    \end{align}
To complete the proof modulo 2187, we set
\begin{equation*}
F(q):=\sum_{n=0}^{\infty}c\phi_{18}(3n+2)q^n.
\end{equation*}
Then \eqref{Gen 3n+2 Mod 2187 Eve4} gives
\begin{align*}
F(q)\equiv {}&
3^4\cdot13
\frac{f_2^5f_3^3f_6^{10}}
     {f_1^{13}f_4^2f_{12}^4}
-
3^4
\frac{f_2^{18}f_3^8}
     {f_1^{20}f_4^6f_6}
+
3^4q
\frac{f_3^4f_4^2f_6^5}
     {f_1^8f_2^4}
-
3^5q^2
\frac{f_2^7f_3^7f_{12}^4}
     {f_1^{17}f_4^2}
\\
&-
2^2\cdot3^5
\frac{f_2^{12}f_3^7f_6}
     {f_1^{17}f_4^3f_{12}}
+
2\cdot3^5q
\frac{f_2^6f_3^7f_6^7}
     {f_1^{17}f_4^3f_{12}}
+
2^2\cdot3^5q
\frac{f_2^{14}f_3^{12}f_6^3}
     {f_1^{24}f_4^6}
+
3^5\cdot7q^2
\frac{f_2^8f_3^{14}f_{12}^4}
     {f_1^{22}f_4^2f_6^3}
\\
&+
3^6q
\frac{f_3^{10}f_6^{12}}
     {f_1^{10}f_2f_4^{11}f_{12}}
-
3^6q^2
\frac{f_3^8f_6^{13}}
     {f_1^{20}f_{12}^2}
+
3^6q^3
\frac{f_3^6f_6f_{12}^6}
     {f_1^{14}}
\pmod{2187}.
\end{align*}
Thus
\begin{equation*}
F(q)\equiv81H(q)\pmod{2187},
\end{equation*}
where
\begin{align*}
H(q)
=&
13
\frac{f_2^5f_3^3f_6^{10}}
     {f_1^{13}f_4^2f_{12}^4}
-
\frac{f_2^{18}f_3^8}
     {f_1^{20}f_4^6f_6}
+
q
\frac{f_3^4f_4^2f_6^5}
     {f_1^8f_2^4}
-
3q^2
\frac{f_2^7f_3^7f_{12}^4}
     {f_1^{17}f_4^2}
-
12
\frac{f_2^{12}f_3^7f_6}
     {f_1^{17}f_4^3f_{12}}
     \\
&+
6q
\frac{f_2^6f_3^7f_6^7}
     {f_1^{17}f_4^3f_{12}}
+
12q
\frac{f_2^{14}f_3^{12}f_6^3}
     {f_1^{24}f_4^6}
+
21q^2
\frac{f_2^8f_3^{14}f_{12}^4}
     {f_1^{22}f_4^2f_6^3}
+
9q
\frac{f_3^{10}f_6^{12}}
     {f_1^{10}f_2f_4^{11}f_{12}}
\\
&-
9q^2
\frac{f_3^8f_6^{13}}
     {f_1^{20}f_{12}^2}
+
9q^3
\frac{f_3^6f_6f_{12}^6}
     {f_1^{14}}.
\end{align*}

We continue to use the above $(p,k)$-parametrizations, where
\begin{equation*}
p=
2q\frac{f_2^3f_3^3f_{12}^6}
        {f_1f_4^2f_6^9},
\qquad
k=
\frac{f_1^2f_4^2f_6^{15}}
     {f_2^5f_3^6f_{12}^6}.
\end{equation*}

We define two auxiliary $q$-product quotients $M$ and $R$ as
\begin{equation*}
M:=
\frac{f_2^5f_3^3f_6^{10}}
     {f_1^{13}f_4^2f_{12}^4},
\qquad
R:=
\frac{f_2^4f_3^3f_{12}^3}
     {f_1f_4f_6^8}.
\end{equation*}


We need the following lemmas to complete the proof.
\begin{lemma}\label{lem:H-decomp}
We have
\begin{equation*}
H(q)=M\left(13+RA(p)+R^2B(p)\right),
\end{equation*}
where
\begin{align*}
A(p)
={}&
\frac{2p+1}{p-1}
-
\frac{12p(p+1)^2}
     {(p-1)^3(p+2)}
-
\frac{21p^2(p+1)^2}
     {4(p-1)^3(2p+1)}
+
\frac{144p(p+1)^4}
     {(p+2)^5(2p+1)^2}
\end{align*}
and
\begin{align*}
B(p)
={}&
\frac{p}{2(2p+1)^2}
-
\frac{3p^2}
     {4(p-1)^2(2p+1)}
+
\frac{12}{p-1}
+
\frac{6p(p+1)}
     {(p-1)^2(p+2)(2p+1)}
\\
&-
\frac{9p^2(p+1)^2}
     {2(p-1)^4(p+2)(2p+1)^2}
+
\frac{9p^3}
     {8(p-1)^2(2p+1)^2}.
\end{align*}
\end{lemma}

\begin{proof}
We set
\begin{equation}\label{h(q)}
H(q):=13E_1-E_2+E_3-3E_4-12E_5+6E_6+12E_7+21E_8+9E_9-9E_{10}+9E_{11}.
\end{equation}
Here $E_1=M$. Using the $(p,k)$-parametrization, we have
\begin{align*}
\frac{E_1}{M}&=1,\\ \\
\frac{E_2}{M}
&= \frac{f_2^{13}f_3^5f_{12}^4}
     {f_1^7f_4^4f_6^{11}} =
2^{-1/3}
(1-p)^{-2/3}
(1+2p)^{4/3}
(2+p)^{1/3}
\\
&=
-R\frac{2p+1}{p-1},\\ \\
\frac{E_3}{M} &= q\,
\frac{f_1^5f_3f_4^4f_{12}^4}
     {f_2^9f_6^5}
=
2^{-5/3}p
(1-p)^{2/3}
(1+2p)^{-4/3}
(2+p)^{2/3}
\\
&=
R^2\frac{p}{2(2p+1)^2},\\ \\
\frac{E_4}{M} &= q^2
\frac{f_2^2f_3^4f_{12}^8}
     {f_1^4f_6^{10}}
=
2^{-8/3}p^2
(1-p)^{-4/3}
(1+2p)^{-1/3}
(2+p)^{2/3}
\\
&=
R^2
\frac{p^2}{4(p-1)^2(2p+1)},\\ \\
\frac{E_5}{M} &= \frac{f_2^7f_3^4f_{12}^3}
     {f_1^4f_4f_6^9}
=
2^{-2/3}
(1-p)^{-1/3}
(1+2p)^{2/3}
(2+p)^{2/3}
\\
&=
-\frac{R^2}{p-1},\\ \\
\frac{E_6}{M} &= q\,
\frac{f_2f_3^4f_{12}^3}
     {f_1^4f_4f_6^3}
=
2^{-2/3}p(1+p)
(1-p)^{-4/3}
(1+2p)^{-1/3}
(2+p)^{-1/3}
\\
&=
R^2
\frac{p(p+1)}
     {(p-1)^2(p+2)(2p+1)},\\ \\
\frac{E_7}{M} &= q\,
\frac{f_2^9f_3^9f_{12}^4}
     {f_1^{11}f_4^4f_6^7}
=
2^{-1/3}p(1+p)^2
(1-p)^{-8/3}
(1+2p)^{1/3}
(2+p)^{-2/3}
\\
&=
-R
\frac{p(p+1)^2}
     {(p-1)^3(p+2)},\\ \\
\frac{E_8}{M} &= q^2
\frac{f_2^3f_3^{11}f_{12}^8}
     {f_1^9f_6^{13}}
=
2^{-7/3}p^2(1+p)^2
(1-p)^{-8/3}
(1+2p)^{-2/3}
(2+p)^{1/3}
\\
&=
-R
\frac{p^2(p+1)^2}
     {4(p-1)^3(2p+1)},\\ \\
\frac{E_9}{M} &= q\,
\frac{f_1^3f_3^7f_6^2f_{12}^3}
     {f_2^6f_4^9}
=
2^{11/3}p(1+p)^4
(1-p)^{1/3}
(1+2p)^{-5/3}
(2+p)^{-14/3}
\\
&=
R
\frac{16p(p+1)^4}
     {(p+2)^5(2p+1)^2},\\ \\
\frac{E_{10}}{M} &= q^2
\frac{f_3^5f_4^2f_6^3f_{12}^2}
     {f_1^7f_2^5}
=
2^{-5/3}p^2(1+p)^2
(1-p)^{-10/3}
(1+2p)^{-4/3}
(2+p)^{-1/3}
\\
&=
R^2
\frac{p^2(p+1)^2}
     {2(p-1)^4(p+2)(2p+1)^2},\\ \\
\frac{E_{11}}{M} &= q^3
\frac{f_3^3f_4^2f_{12}^{10}}
     {f_1f_2^5f_6^9}
=
2^{-11/3}p^3
(1-p)^{-4/3}
(1+2p)^{-4/3}
(2+p)^{2/3}
\\
&=
R^2
\frac{p^3}
     {8(p-1)^2(2p+1)^2}.
\end{align*}

Substituting the above values in \eqref{h(q)} gives
\begin{equation*}
H(q)=M\left(13+RA(p)+R^2B(p)\right),
\end{equation*}
which proves the lemma.
\end{proof}


\begin{lemma}\label{lem:R-cubic}
Let $a:=1-p$, then
\begin{equation*}
R^3=a^3+\frac92pa.
\end{equation*}
\end{lemma}

\begin{proof}
Using the $(p,k)$-parametrizations, we have
\begin{align*}
R^3&=\frac{(1-p)(1+2p)(2+p)}{2}=\frac{a(2+5p+2p^2)}{2}=a^3+\frac92pa,
\end{align*}
which completes the proof.
\end{proof}


\begin{lemma}\label{lem:R-mod9}
Let $a:=1-p$, then
\begin{equation*}
R\equiv a+\frac{3p}{2a}\pmod9.
\end{equation*}
\end{lemma}

\begin{proof}
By Lemma \ref{lem:R-cubic},
\begin{equation*}
R^3\equiv a^3\pmod3,
\end{equation*}
which by the Fermat's Little Theorem \cite[Page 88]{burton2010ebook} gives
\begin{align*}
    R\equiv a\pmod3.
\end{align*}
Let
\begin{equation}
R=a+3T.\label{R=a+3t}
\end{equation}
Then
\begin{equation*}
(a+3T)^3=a^3+\frac92pa
\end{equation*}
and hence
\begin{equation*}
a^2T+3aT^2+3T^3
=
\frac{pa}{2}.
\end{equation*}
Working modulo $3$ in the above, we obtain
\begin{equation*}
a^2T\equiv\frac{pa}{2}\pmod3,
\end{equation*}
which gives
\begin{equation*}
T\equiv\frac{p}{2a}\pmod3.
\end{equation*}
We write
\begin{equation*}
T:=\frac{p}{2a}+3U.
\end{equation*}
Therefore \eqref{R=a+3t} becomes
\begin{equation*}
R=a+\frac{3p}{2a}+9U,
\end{equation*}
which modulo 9 proves the lemma.
\end{proof}


\begin{lemma}\label{lem:R-mod27}
Let
\begin{align*}
\rho(p)
&=
a+\frac{3p}{2a}
-\frac{9p^2(2a^2+p)}{8a^5},\\
Q(p)
&=
8p^6-36p^5+54p^4-61p^3
+54p^2-36p+8.
\end{align*}
Then
\begin{align*}
R&\equiv\rho(p)\pmod{27},\\
\rho(p)
&=
-\frac{Q(p)}{8(p-1)^5}.
\end{align*}
\end{lemma}

\begin{proof}
Suppose that $\displaystyle{T_0:=\frac{p}{2a}}$, then $\displaystyle{a^2T_0=\frac{pa}{2}}$ and from the proof of Lemma \ref{lem:R-mod9},
\begin{equation*}
T=T_0+3U
\end{equation*}
and
\begin{equation*}
a^2(T_0+3U)
+
3a(T_0+3U)^2
+
3(T_0+3U)^3
=
\frac{pa}{2},
\end{equation*}
which gives
\begin{equation*}
a^2U
+
a(T_0+3U)^2
+
(T_0+3U)^3
=
0,
\end{equation*}
which under modulo $3$ reduces to
\begin{equation*}
a^2U+aT_0^2+T_0^3
\equiv0\pmod3.
\end{equation*}
Therefore
\begin{align*}
U&\equiv
-\frac{p^2(2a^2+p)}
       {8a^5}
\pmod3.
\end{align*}
Thus
\begin{equation*}
9U
\equiv
-\frac{9p^2(2a^2+p)}
       {8a^5}
\pmod{27}.
\end{equation*}
From the proof of the last lemma and the above, we arrive at
\begin{equation*}
R
\equiv
a+\frac{3p}{2a}
-\frac{9p^2(2a^2+p)}
       {8a^5}
\pmod{27},
\end{equation*}
which proves
\begin{equation*}
R\equiv\rho(p)\pmod{27}.
\end{equation*}

Since $a=1-p$,
\begin{align*}
\rho(p)
&=
1-p
+
\frac{3p}{2(1-p)}
-
\frac{9p^2(2(1-p)^2+p)}
       {8(1-p)^5}
\\
&=
-\frac{
8p^6-36p^5+54p^4-61p^3
+54p^2-36p+8
}
{8(p-1)^5},
\end{align*}
which proves
\begin{equation*}
\rho(p)
=
-\frac{Q(p)}{8(p-1)^5}.
\end{equation*}
\end{proof}


\begin{lemma}\label{lem:main-pk-cong}
We have
\begin{equation*}
13+RA(p)+R^2B(p)\equiv0\pmod{27},
\end{equation*}where $A(p)$ and $B(p)$ are defined in Lemma \ref{lem:H-decomp}.
\end{lemma}

\begin{proof}
Using Lemma \ref{lem:R-mod27}, we have
\begin{equation*}
RA(p)\equiv\rho(p)A(p)\pmod{27}
\end{equation*}
and
\begin{equation*}
R^2B(p)\equiv\rho(p)^2B(p)\pmod{27}.
\end{equation*}
Therefore
\begin{equation*}
13+RA(p)+R^2B(p)
\equiv
13+\rho(p)A(p)+\rho(p)^2B(p)
\pmod{27}.
\end{equation*}

We set
\begin{equation*}
A(p):=-\frac{P_1(p)}{4(p-1)^3(p+2)^5(2p+1)^2},
\end{equation*}
where
\begin{align*}
P_1(p)
={}&
10p^{10}
+413p^9
+3270p^8
+16161p^7
+42570p^6
\\
&+
60996p^5
+49224p^4
+22512p^3
+5664p^2
+512p
-128,
\end{align*}
and
\begin{equation*}
B(p):=\frac{P_2(p)}{8(p-1)^4(p+2)(2p+1)^2},
\end{equation*}
where
\begin{equation*}
P_2(p)
=
385p^6
+82p^5
-1523p^4
+500p^3
+836p^2
-232p
-192.
\end{equation*}
Therefore, using the last lemma, we obtain
\begin{align*}
13+\rho A+\rho^2B
={}&
13
+
\frac{Q(p)P_1(p)}
{32(p-1)^8(p+2)^5(2p+1)^2}
\\
&+
\frac{Q(p)^2P_2(p)}
{512(p-1)^{14}(p+2)(2p+1)^2}.
\end{align*}

We take $D(p):=512(p-1)^{14}(p+2)^5(2p+1)^2$, then
\begin{align*}
D(p)\left(13+\rho A+\rho^2B\right)
&=
6656(p-1)^{14}(p+2)^5(2p+1)^2
\\
&\quad
+
16(p-1)^6Q(p)P_1(p)
+
(p+2)^4Q(p)^2P_2(p)
\\
&=:N(p).
\end{align*}
Simplifying,
\begin{align*}
N(p)
=
27p\Bigl(
&960p^{21}
+1728p^{20}
-22512p^{19}
+1440p^{18}
+61860p^{17}
\\
&+
236132p^{16}
-670065p^{15}
-34970p^{14}
+1807995p^{13}
\\
&-
2493752p^{12}
-33448p^{11}
+4648744p^{10}
-5672512p^9
\\
&-
301344p^8
+6442432p^7
-5088768p^6
-109312p^5
\\
&+
1989120p^4
-827392p^3
-32768p^2
+69632p
-8192
\Bigr).
\end{align*}
Therefore
\begin{equation*}
D(p)\left(13+\rho A+\rho^2B\right)
\equiv0\pmod{27}.
\end{equation*}

Observe that
\begin{equation*}
D(0)
=
512(-1)^{14}(2)^5
=
2^{14},
\end{equation*}
which gives
\begin{equation*}
3\nmid D(0).
\end{equation*}
Therefore, $D(p)$ is a unit in $\mathbb{Z}_3[[p]]$. Thus
\begin{equation*}
13+\rho A+\rho^2B
\equiv0\pmod{27},
\end{equation*}
which proves the lemma.
\end{proof}

By Lemma \ref{lem:H-decomp}, we have
\begin{equation*}
H(q)
=
M\left(13+RA(p)+R^2B(p)\right).
\end{equation*}
Again, by Lemma \ref{lem:main-pk-cong}, we have
\begin{equation*}
13+RA(p)+R^2B(p)
\equiv0\pmod{27}.
\end{equation*}
Therefore
\begin{equation*}
H(q)\equiv0\pmod{27}.
\end{equation*}
As a consequence,
\begin{align*}
F(q)&\equiv81H(q)\pmod{2187}\\
&\equiv 0\pmod{2187}.
\end{align*}
Thus
\begin{equation*}
\sum_{n=0}^{\infty}
c\phi_{18}(3n+2)q^n
\equiv0\pmod{2187}.
\end{equation*}
Comparing the coefficients, we finally arrive at
\begin{equation*}
c\phi_{18}(3n+2)
\equiv0\pmod{2187}.
\end{equation*}
This completes the proof.
\end{proof}
\section{Concluding Remarks}

    \begin{enumerate}
        \item Based on numerical evidence, we conjecture the following congruences and leave their proofs for future investigation.
    \begin{conjecture}For all $n\ge0$,
       \begin{align*}
        c\phi_{18}\left(30n+19\right)&\equiv 0\pmod{16},\\
			c\phi_{18}\left(30n+25\right)&\equiv 0\pmod{16}.
    \end{align*} 
    \end{conjecture}
    \end{enumerate}
    
\section*{Statements and Declarations}
    
    \textbf{Data Availability} This manuscript does not contain any associated data.\\
	\textbf{Competing Interests} The authors declare that they have no competing interests.\\
	\textbf{Funding Information} No funding was received for the preparation of this manuscript.

	\bibliographystyle{alpha}
	\bibliography{ref}
	
\end{document}